\documentclass{amsart}
\usepackage{amssymb,amsthm,amscd}
\usepackage[pdftex]{graphics}
\usepackage{graphics}
\usepackage{graphicx}
\usepackage[pagebackref]{hyperref}

\newcommand{\F}{{\mathbb F}}
\newcommand{\Z}{{\mathbb Z}}
\newcommand{\bP}{{\mathbb P}}
\newcommand{\Q}{{\mathbb Q}}
\newcommand{\R}{{\mathbb R}}
\newcommand{\C}{{\mathbb C}}
\newcommand{\cC}{{\mathcal C}}

\newcommand{\rad}{{\operatorname{rad}\,}}
\newcommand{\lra}{\longrightarrow}
\newcommand{\too}{\longmapsto}

\newtheorem{thm}{Theorem}
\newtheorem{prop}{Proposition}
\newtheorem{lem}{Lemma}
\newtheorem{cor}{Corollary}

\title{Parametrizing Algebraic Curves}
\author{F. Lemmermeyer}
\address{M\"orikeweg 1, 73489 Jagstzell}
\email{hb3@ix.urz.uni-heidelberg.de}

\begin{document}

\begin{abstract}
We present the technique of parametrization of plane algebraic curves
from a number theorist's point of view and present Kapferer's simple 
and beautiful (but little known) proof that nonsingular curves of degree 
$> 2$ cannot be parametrized by rational functions.
\end{abstract}

\maketitle
\begin{center} \today  \end{center}

\section{Introduction}

The parametrization of plane algebraic curves (or, more 
generally, of algebraic varieties) is an important tool for 
number theorists. In this article we will briefly sketch some
background, give a few applications, and then point out
the limits of the method determined by Clebsch's Theorem
according to which curves can be parametrized by rational 
functions if and only if their genus is $0$. The main contribution
of this article is a presentation of Kapferer's simple proof of 
the following special case of Clebsch's theorem: nonsingular 
curves of degree $\ge 3$ cannot be parametrized by rational 
functions. As we will see, the similarity with Fermat's 
Last Theorem is more than a pure accident.

The material presented below can be used for giving undergraduate
introductions to projective geometry or algebraic geometry\footnote{As
for introductions to projective geometry, my personal favorite is 
Samuel's \cite{Samuel}. Very nice introductions to algebraic curves 
are Reid's \cite{Reid} or Fulton's classical \cite{Ful}.} an 
arithmetic touch. Those who would like to study modern techniques
for parametrizing curves and varieties should consult the work
of Winkler and his coauthors; see e.g. \cite{SWP}.

\section{From Calculus to Multiplicity}
Computing tangents to a smooth function is an easy exercise in
elementary calculus. Given the problem of finding the tangent 
to the function $y = f(x) = x^4 + x + 1$ at $x = 0$, most 
students probably would start computing the derivative
$f'(x)$ and follow the routines that they have memorized. 
Those who remember that tangents are supposed to be linear
approximations of $f$ might be able to guess that the equation
of the tangent must be $y = x+1$, i.e. the linearization of 
$y = f(x)$: in fact, for $x \approx 0$, the term $x^4$ is 
very small compared to $x+1$.

This simple observation suffices for defining and computing the 
tangents to any polynomial function not just over the reals as
in calculus but over any field of interest to number theorists. 
For finding the tangent to $y = f(x) = x^n$ in $x = a$, consider 
the function $g(x) = f(x+a) = (x+a)^n$ at $x = 0$; since 
$g(x) = x^n + \ldots + \binom{n}{1}xa^{n-1} + a^n$, the tangent 
to $g$ in $x = 0$ is $y = na^{n-1}x+a^n$, hence the tangent to 
$f$ in $x = a$ is $y = na^{n-1}(x-a)+a^n$. In particular, the 
slope of the tangent to $y = f(x) = x^n$ is $f'(a) = na^{n-1}$.

The same method works for tangents to general algebraic curves:
for finding the tangent to the unit circle $X^2 + Y^2 = 1$
at $(x,y) = (1,0)$, we shift the coordinate system and consider
$(X+1)^2 + Y^2 = 1$ at $(0,0)$, giving the tangent $X = 0$.
Thus the tangent to the unit circle in $(1,0)$ is $X = 1$.

Does this method work for all curves and all points? The answer 
is no: consider the curve $F(X,Y) = 0$ and a point $P(a,b)$; then 
we have to look at $G(X,Y) = F(X+a,Y+b)$ at the origin. Since 
$G(0,0) = 0$, the polynomial $G(X,Y)$ has no constant term.
The linear terms of $G$ define a tangent unless there are
no linear terms at all. This may happen, as the example
$G(X,Y) = Y^2 - X^3 - X^2$ shows\footnote{In this case,
neglecting the cubic terms we get $Y^2 = X^2$, which is the
pair of lines $Y = X$ and $Y = -X$. In fact, the curve 
defined by $g$ has ``two'' tangents at the origin. Similarly,
neglecting the higher terms of the cubic $Y^2 = X^3$ we get 
$Y^2 = 0$, indicating that this cubic has $Y = 0$ as a ''double
tangent'' at the origin. In general, the quadratic terms
$aX^2 + bXY + cY^2$ factor over some quadratic extension of
the base field; in plots of the curves, the two tangents are
visible only if the quadratic extension is real. As an example
of two imaginary tangents, consider the cubic $Y^2 = X^3 - X^2$:
over the reals, this curve has an isolated point at the origin}. 
In such a case, we say that the curve defined by $g$ is singular 
at the origin, or that the curve defined by $f$ is 
singular at $P$. The multiplicity of the singular origin on $g$ is 
the smallest degree of any term occurring in the equation of the 
curve. Curves of high degree 
can have many singular points. Clearly a curve $F(X,Y) = 0$
with $F(0,0) = 0$ is singular at the origin if and only if the
partial derivatives $\frac{\partial F}{\partial X} = F_X$
and $\frac{\partial F}{\partial Y} = F_Y$ both vanish at 
$(0,0)$. It is a simple exercise to show that the curve defined
by $F(X,Y) = 0$ is singular at the affine point $P = (x,y)$ if and
only if $F(P) = F_X(P) = F_Y(P) = 0$.

It remains to bring in the points at infinity by switching to 
the projective point of view. The affine curve $\cC:y^2 = x^4 + 1$ 
can be ''projectivized'' by 
homogenizing the defining equation; the projective closure of $\cC$
is the projective curve defined by $Y^2Z^2 = X^4 + Z^4$. The affine
curve does not have a singular point in the affine plane; its projective 
closure $Y^2Z^2 = X^4 + Z^4$ has the point $[0:1:0]$ at infinity, which 
is easily seen to be singular since the partial derivatives of
$F(X,Y,Z) = Y^2Z^2 - X^4 - Z^4$ all vanish at this point.
 
Extending the affine criterion of singularity to the projective situation
is straight forward:

\begin{thm}
Let $K$ be an algebraically closed field. A point $P$ in the projective 
plane over $K$ is a singular point of the projective curve
defined by $F(X,Y,Z) = 0$ if and only if $F_X(P) = F_Y(P) = F_Z(P) = 0$.
\end{thm}

A very simple construction of singular curves is based on the observation
that points in which a curve intersects itself\footnote{See e.g. the curve
in Fig. \ref{FP3}.} are necessarily singular. In particular, all curves 
$F(X,Y) = 0$ with $F = GH$ for nonconstant polynomials $G, H \in K[X,Y]$ 
have singularities at the points where the curves $G(X,Y) = 0$ and 
$H(X,Y) = 0$ intersect. The number of such intersection points is determined
by Bezout's Theorem. In its weakest form it says that two curves of degree 
$m$ and $n$ and without a common component intersect in at most $mn$ 
points; the strong version claims that if the points are counted with 
proper multiplicity, over an algebraically closed field and in the 
projective plane, then there are exactly $mn$ points of intersection.

Bezout's Theorem can be used to classify singular curves
with small degree:

\begin{itemize}
\item Singular conics\footnote{Conics (short for conic sections) are curves 
      of degree $2$; over the reals, nonsingular conics are ellipses, 
      hyperbolas, and parabolas.} are reducible, i.e., consist of two lines.
\item Cubics with two singular points are reducible: the line through 
      two singular points intersects the cubic in four points (counted
      with multiplicity), hence the cubic must contain this line by 
      Bezout's Theorem. 
\item In particular, irreducible cubics can have at most one singular 
      point\footnote{If the base field $F$ is perfect, this singular point 
      necessarily has coordinates in $F$ by Galois theory. Over $\F_2[T]$, the 
      curve $y^2 = x^3 + T$ is singular in $(0,\sqrt{T}\,)$; its coordinates 
      lie in an inseparable quadratic extension of $F$.}; if there were 
      two of them, the line through these two points would intersect the 
      cubic with multiplicity $4 > 1 \cdot 3$, which is only possible if 
      the cubic contains the line, i.e., is reducible.
\end{itemize}

\section{Parametrizing Conics}

One of the most classical diophantine problems is the construction
of Pythagorean triples: these are points $(x,y,z)$ with integral
coordinates lying on the projective curve $X^2 + Y^2 = Z^2$. 
The Pythagorean equation describes the projective closure of the unit 
circle $x^2 + y^2 = 1$. 

\begin{figure}[!ht]
\begin{center}
\includegraphics[width=5.2cm,height=4cm]{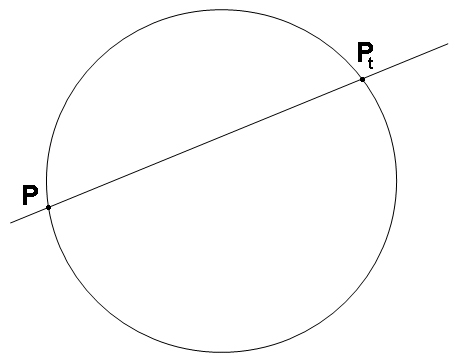}
\end{center}
\caption{Parametrizing the Unit Circle}\label{FP1}
\end{figure}

The geometric method\footnote{The first known geometric parametrization
of an algebraic curve is due to Newton \cite{New}, and is contained in an
article published only in 1971. The nowadays ubiquitous parametrization of
the unit circle first appeared at the beginning of the 20th century in 
textbooks such as Kronecker's \cite{Kro}.} 
for finding a parametrization of the 
rational points\footnote{The notion of a rational point depends
on the base field. Below, we will tacitly assume that the base 
field is $\Q$, and then rational points are points with rational
coordinates. More generally, for curves defined by a polynomial
$f \in K[X,Y]$, a $K$-rational point is a point on the curve with
coordinates in $K$.} on the unit circle (these are points $(x,y)$ 
with $x, y \in \Q$ such that $x^2 + y^2 = 1$) is the following:
given a point such as $Q = (-1,0)$, the lines $y = t(x+1)$ through 
$Q$ intersect the circle in $Q$ and another point $P_t$, which
is easily computed as\footnote{Observe that these formulas show
that the unit circle has a $\Q(t)$-rational point, that is, a
point defined over the rational function field of $\Q$.}  
$$ P_t = \bigg( \frac{1-t^2}{1+t^2}, \frac{2t}{1+t^2} \bigg). $$
Each rational slope gives a rational point, and conversely, every
rational point $P = (x,y) \ne Q$ on the unit circle has the form 
$P = P_t$ for $t = \frac{y}{x+1}$. Thus the parametrization 
$\Q \lra \cC \setminus \{Q\}: t \too P_t$ is a bijection between 
the affine line over $\Q$ and the conic (minus $Q$). Allowing
$t = \infty$ gives a bijection between the projective line
$\bP^1\Q$ and $\cC(\Q)$. Setting $t = \frac mn$ then provides us
with the parametrization
$$ x = m^2 - n^2, \quad y = 2mn, \quad z = m^2 + n^2 $$
of Pythagorean triples, i.e., integral solutions of the
equation $x^2 + y^2 = z^2$.

The same argument goes through for all irreducible conics: if we know
a single point $P$ on a conic $\cC$, we can find all of them by 
intersecting $\cC$ with lines through $P$. Some conics, such as
$x^2 + y^2 = 3$, do not have any points over certain fields like
$\Q$ or $\F_3$; for parametrizing them, we have to find a point
over an extension field (such as $\Q(\sqrt{3}\,)$, $\Q(i)$ or
$\F_9$), and the formulas giving the parametrization then will
involve certain irrationals.

The parametrization of conics can be used for solving
a variety of problems. The Arabs already knew how to 
find infinitely many integral solutions of the equation 
$x^4 + y^2 = z^2$ by solving $X^2 + y^2 = z^2$ and 
showing that $X = x^2$ infinitely often. Their unability 
of solving the similar equation $x^4 + y^4 = z^2$ made
them conjecture that there are no integral solutions; 
Fermat and Euler later found full proofs.

The rational parametrization of the unit circle can be used to 
transform integrals of the type $\int \frac{dx}{\sqrt{1-x^2}}$
into integrals of rational functions. The substitution
$x = \frac{t^2-1}{t^2+1}$ gives $y = \sqrt{1-x^2} = \frac{2t}{t^2+1}$
(with positive $t$ if we take the square root to be positive), hence
$$ \int_{-1}^1 \frac{dx}{\sqrt{1-x^2}} 
       = 2\int_0^\infty \frac{dt}{t^2+1}. $$

Euler also showed how to use the parametrization of conics in 
solving $y^2 = x^3+1$: shifting the equation by $x = z-1$ give
$y^2 = z(z^2 - 3z + 3)$; the factors on the right have greatest
common divisor $\gcd(z,z^2 - 3z + 3) = \gcd(z,3)$, hence are coprime
or have gcd $3$. Unique factorization implies that the factors are 
either squares or three times squares (up to sign). The case $z = r^2$ 
leads, for example, to the quartic curve $y^2 = r^4 - 3r^2 + 3$. By 
studying these quartics arising from $y^2 = x^3 + 1$ Euler found all 
rational points on this elliptic curve (see \cite{LPep} for an exposition
of Euler's proof).

More generally, solving equations such as $y^2 = x(x^2 + ax + b)$
over the rationals (that is, finding rational points on the
elliptic curve $E: y^2 = x(x^2 + ax + b)$ with discriminant
$\Delta = 16b^2(a^2-4b)$ and its dual\footnote{The correct terminology
is "isogenous": there are isogenies $E \lra E'$ and $E' \lra E$ whose
composition is the map $E \lra E$ induced by multiplication by $2$.} 
curve $E': y^2 = x(x^2 + a' x + b')$ with $a' = -2a$ and $b' = a^2 - 4b$) 
inevitably leads to the problem of deciding whether the finitely many 
equations 
\begin{equation}\label{Etors}
b_1m^4 + am^2n^2 + b_2n^4 = e^2,
\end{equation} 
where $b_1b_2 = b$, have nontrivial solutions in the rationals. A 
necessary condition for solvability in the rationals is solvability 
in all completions of the rationals. This condition can be checked in 
finitely many steps thanks to the following result, which can also
be proved using the parametrization of conics (see \cite{AL}):

\begin{prop}
The equation $b_1m^4 + am^2n^2 + b_2n^4 = e^2$ has nontrivial 
solutions in the $p$-adic integers $\Z_p$ for all primes 
$p \nmid 2ab_1b_2(a^2 - 4b_1b_2)$.
\end{prop}

This means that for checking the solvability of 
$b_1m^4 + am^2n^2 + b_2n^4 = e^2$ in all completions
of $\Q$, we only have to look at $\R = \Q_\infty$ 
and the finitely many $p$-adic fields $\Q_p$ for primes 
$p \mid 2ab_1b_2(a^2 - 4b_1b_2)$.

The idea behind the proof is quite simple: first show that 
the conic  $b_1x^2 + axy + b_2y^2 = z^2$ has a nontrivial 
solution in $\F_p$; using this point, parametrize the conic
to find all of them, and then show that there is a solution 
$(x,y)$ for which $x = m^2$ and $y = n^2$ are both squares.

\medskip

Of course the line of proof could be simplified drastically 
if we were able to simply write down a parametrization of 
(\ref{Etors}). But there are two obstructions: for parametrizing 
a curve, we need a rational point to start with (whose existence 
we would like to prove in the first place). And even if we had 
such a rational point we would not be able to find such a 
parametrization: as we will see below, curves such as 
(\ref{Etors}) cannot be parametrized.

\section{Parametrizing Curves of Higher Degree.}

Conics are not the only curves that can be parametrized.
In fact, we can start with any "parametrization", say
$$ x = \frac{t(t^2+1)}{t^4+1}, \quad y = \frac{t(t^2-1)}{t^4+1} $$
and then find the equation of the corresponding plane algebraic 
curve by eliminating\footnote{This can be achieved easily by
using resultants. The {\tt pari} command 
\begin{center}
{\tt polresultant}(($t^4+1)*x-t*(t^2+1),(t^4+1)*y-t*(t^2-1),t$)
\end{center}
produces the equation $4(x^4 + 2x^2y^2 + y^4 - x^2 + y^2) = 0$.
See Prop. \ref{Plem} below.}
$t$ from these equations.

The cubics with a singularity at the origin have the 
form\footnote{after a suitable projective transformation}
$y^2 = x^3 + ax^2$; the cubic $y^2 = x^3 + x^2$, for example, 
has a singular point at the origin $O$, and using lines through 
$O$ we find the parametrization $x = t^2-1$, $y = t^3-t$.

\begin{figure}[!ht]
\begin{center}
\includegraphics[width=3cm,height=3cm]{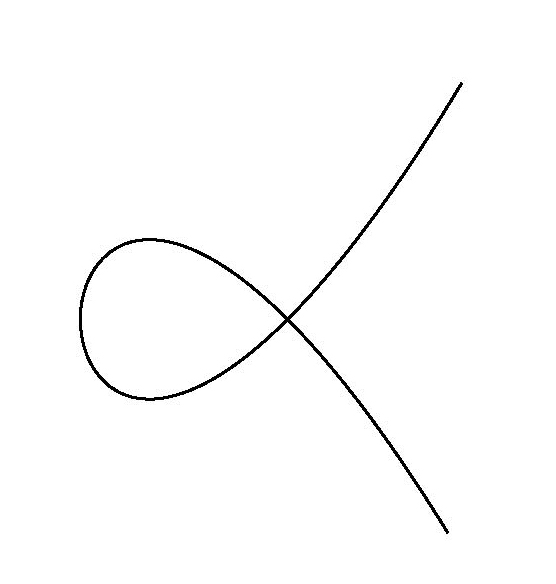}
\end{center}
\caption{The singular cubic $y^2 = x^3 + x^2$}
\label{FP3}
\end{figure}

Just as the parametrization of conics has applications to calculus,
so does the parametrization of cubics. For example, this technique 
allows us to compute the area of the region enclosed by the curve 
$y^2 = x^3 + x^2$: we have
$$ A = 2 \int_{-1}^0 y \, dx = \int_0^1 2t(t-t^3) \, dt = \frac{8}{15}. $$

Absilutely irreducible Curves with degree $n$ having a singularity 
of multiplicity $n-1$ can be parametrized by using a pencil of lines 
through the singularity. In fact, assume that the curve 
$\cC: F(X,Y) = 0$ is defined by an equation 
$F(X,Y) = F_n(X,Y) + F_{n-1}(X,Y)$,
where $F_m$ denotes a polynomial in which each term has degree $m$
(for example, $F(X,Y) = X^4 - X^2Y^2 + Y^3$ can be written as
$F(X,Y) = F_4(X,Y) + F_3(X,Y)$ with $F_4(X,Y) = X^4 - X^2Y^2$ and
$F_3(X,Y) = Y^3$). Plugging the line equation $Y = tX$ into $F(X,Y) = 0$
gives
$$  0 = F_n(X,tX) + F_{n-1}(X,tX) = X^n F_n(1,t) + X^{n-1}F_{n-1}(1,t). $$
The solution $X = 0$ gives the singular point; the nonzero solution
gives the following well-known\footnote{See e.g. Samuel 
\cite[Sect. 2.6]{Samuel}.} parametrization:

\begin{prop}
Let the curve $\cC: F(X,Y) = 0$ be defined by a polynomial
$F(X,Y) = F_n(X,Y) + F_{n-1}(X,Y)$, where $F_m$ denotes a
polynomial in which each term has degree $m$; then 
$$ X = - \frac{F_{n-1}(1,t)}{F_n(1,t)},  \qquad
   Y = - t \cdot \frac{F_{n-1}(1,t)}{F_n(1,t)} $$
is a parametrization of $\cC$.
\end{prop}

Finding a parametrization of the lemniscate $(X^2+Y^2)^2 = X^2-Y^2$ 
is more difficult. It is a straightforward exercise to compute the
singular points of the lemniscate: it has three double points, one at
the origin $[0:0:1]$ and two conjugate singular points $[1:\pm i:0]$ at 
infinity. Each of the points $[1:\pm i:0]$ lies on the projective 
closure of the circle $(x-a)^2 + (y-b)^2 = c$. The circles going 
through the origin $O$ whose tangent in $O$ is $Y=X$ have equation 
$X^2 + Y^2 + tX - tY = 0$.

\begin{figure}[!ht]
\begin{center}
\includegraphics[width=5.2cm,height=4cm]{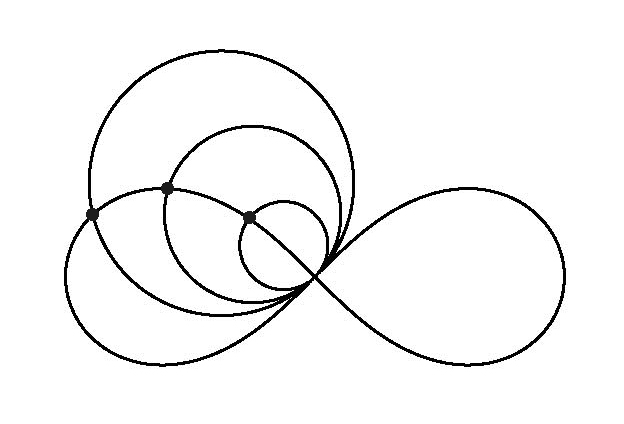}
\end{center}
\caption{The lemniscate and three circles ($t = 0.25$, $0.5$, $0.75$)}
\label{FP5}
\end{figure}

These circles intersect the lemniscate with multiplicity $2$ in 
$[1:\pm i:0]$ and with multiplicity $\ge 3$ in $O$, and since 
there are exactly $8 = 2 \cdot 4$ points of intersection by Bezout's
Theorem, the circles will intersect the lemniscate in exactly one other 
point, whose coordinates will be rational if $t$ is rational. Thus we can 
parametrize the lemniscate using the pencil $X^2 + Y^2 + tX - tY = 0$ 
of circles. Here are the calculations (we use affine coordinates): 
substitute $X^2+Y^2 = t(Y-X)$ in the equation of the lemniscate; this 
gives
$$ 0 = t^2(X-Y)^2 - (X^2-Y^2) = (X-Y)[t^2(X-Y) - (X+Y)] = 0. $$ 
The first factor leads to the known point $O$; setting the second
factor equal to $0$ yields $X(t^2-1) = Y(t^2+1)$. Solving for $Y$
and plugging this into the equation of the circle gives a 
quadratic equation in $X$ without constant term; the nonzero
solution gives the parametrization of the lemniscate:

\begin{prop}\label{Plem}
The lemniscate $(X^2+Y^2)^2 = X^2-Y^2$ admits the parametrization
$$ X(t) = \frac{t(t^2+1)}{t^4+1}, \quad Y(t) = \frac{t(t^2-1)}{t^4+1}. $$
\end{prop}

In fact, any irreducible quartic $F(X,Y,Z) = 0$ with three double 
points can be parametrized (see \cite{Samuel}): move the singular 
points to $[0:0:1]$, $[0:1:0]$ and $[1:0:0]$ by a suitable projective
transformation; the fact that these points are singular implies that 
$F$ has the form
$$ F(X,Y,Z) = aX^2Y^2 + bY^2Z^2 + cX^2Z^2 + dXYZ^2 + eYZX^2 + fXZY^2. $$
Applying the quadratic transformation
$$ X = \frac1x, \quad Y = \frac1y, \quad Z = \frac1z $$
and clearing denominators we end up with a conic, which can easily be
parametrized by rational functions\footnote{If the conic does not have a 
rational point, we have to choose a point with coordinates in some
quadratic extension of $\Q$, and the resulting parametrization will
involve polynomials with coefficients from that field.}.

\section{Curves without Parametrization}

After having given lots of examples for parametrized families
of rational solutions of certain diophantine equations we now
turn to the problem of showing that certain curves cannot be 
parametrized by rational functions. An effective tool for doing
so is provided by the theorem of Stothers-Mason\footnote{For a 
particularly elegant proof see Snyder \cite{Sny}.}. 

For fields $K$, the polynomial ring $K[T]$ is Euclidean and therefore
a unique factorization domain; thus every nonzero polynomial $A \in K[T]$ 
can be written uniquely as a product $A = p_1(T)^{a_1} \cdots p_r(T)^{a_r}$  
of prime powers. We define the radical $\rad A$ of $A$ as the product 
$\rad A = p_1(T) \cdots p_r(T)$.

\begin{thm}[Stothers-Mason]\label{TStM}
Let $K$ be a field of characteristic $0$. If $A, B, C$ are nonzero 
polynomials in $K[T]$ with $A+B+C = 0$ and $\gcd(A,B,C) = 1$, then 
\begin{equation}\label{EIn}
\max \{\deg A,\deg B, \deg C\} \le \deg \rad ABC - 1. 
\end{equation}
\end{thm}

As a corollary, we find that the curve $x^n + y^n = 1$ cannot be
parametrized for exponents $n > 2$, which is called Fermat's Last 
Theorem for polynomials:

\begin{cor}
The Fermat curve $x^n + y^n = 1$ does not have a nontrivial
$\C(t)$-rational point for $n > 2$.
\end{cor}

\begin{proof}
Assume that the Fermat curve admits a rational parametrization by
nonconstant polynomials with coefficients in $\C$. Clearing 
denominators we find polynomials $x, y, z \in K[T]$ with 
$x(T)^n + y(T)^n - z(T)^n = 0$, which we may assume to be pairwise 
coprime. By Mason's Theorem, we have
$$ \deg x(T)^n \le \deg \rad (xyz) - 1 = s-1, $$ 
where $s$ is the number of distinct roots of $xyz$. Thus 
$s \le \deg xyz$, and therefore 
$$ n \deg x = \deg x^n \le \deg x + \deg y + \deg z - 1. $$
The same inequality holds for $y$ and $z$; adding them gives
$$ n (\deg x + \deg y + \deg z) \le 3(\deg x + \deg y + \deg z) - 3, $$
which implies that $n < 3$.
\end{proof}

Although Fermat's Last Theorem for polynomials follows quite easily,
from Mason's Theorem, it is not clear how to apply it to general 
algebraic curves. On the other hand it is possible to derive quite
strong conditions on e.g. the solvability of polynomial Pell equations:
Consider the equation $X^2 - DY^2 = 1$, where $D \in \C[T]$ 
is a nonconstant polynomial of degree $\deg D > 0$. 
Let $n(D)$ denote the number of distinct zeros of $D$. Then we claim

\begin{prop}
Let $D \in \C[T]$ be a polynomial. If  the equation $X^2 - DY^2 = 1$ 
has a nontrivial solution (i.e. with $Y \ne 0$), then 
$\deg D \le 2 n(D) - 2$.
\end{prop}

\begin{proof}
Applying Thm. \ref{TStM} to $X^2 - DY^2 - 1 = 0$ we get
\begin{align*}
 2 \deg X & = \deg X^2 \le \deg X + \deg Y + n(D) - 1, \\
 2 \deg Y + \deg D & = \deg DY^2 \le \deg X + \deg Y + n(D) - 1.
\end{align*} 
Adding these inequalities shows that if $X^2 - DY^2 - 1 = 0$ 
has a nonzero solution, then $\deg D \le 2 n(D) - 2$. Conversely,
if $\deg D \ge 2 n(D) - 1$, the equation has no nonzero solution.
\end{proof}

\section{Kapferer's Proof}

We have already seen that irreducible conics with a rational point
can always be parametrized. For irreducible cubics, the situation
is also clear: if $\cC$ is a singular irreducible cubic, then it
has a unique singular point $P$, and by intersecting the lines
through $P$ with $\cC$ it is then easy to find a parametrization.
Smooth cubics, on the other hand, cannot be parametrized by rational
functions, as we will see below. For quartics $\cC$, there are already
a lot of cases: an irreducible quartic can have at most three singular 
points (if there are four of them, pick a fifth point on the quartic; 
the conic through these five points then intersects the quartic with
multiplicity $\ge 9$, which implies by Bezout's Theorem that the 
conic is a component of the quartic); if $\cC$ has three double 
points, then it can be parametrized by looking at the family of 
conics going through the three singular points and some fixed 
smooth point on the quartic; each such conic intersects the 
quartic in exactly one other point, which gives the required 
parametrization.

Clebsch showed\footnote{See his book \cite{Clgen} as well as 
Shafarevich's article on the occasion of Clebsch'S 150th 
birthday \cite{ShaCl}.}
that the genus of a plane algebraic curve can be computed as 
follows: using birational transformations (the genus is a 
birational invariant), transform the curve into a curve whose 
only singularities are simple nodes or cusps. If $\cC$ is an 
irreducible plane algebraic curve with at most double points 
as singularities, and if $d$ denotes the degree of $\cC$ and 
$r$ the number of double points, then
$$ g = \frac{(d-1)(d-2)}2 - r$$
is called the genus\footnote{It is not difficult to show that a curve 
with degree $d$ can have at most $\frac{(d-1)(d-2)}2$ double points.} of $\cC$.

Clebsch \cite{Clrat} then proved that a plane algebraic curve can 
be parametrized by rational functions\footnote{This result was later
proved in a more number theoretical context by Hilbert \& Hurwitz 
\cite{HH} as well as by Poincar\'e \cite{Poi}. Observe that while
algebraic geometers might be content with a parametrization by
rational functions whose coefficients lie in some algebraically
closed field such as $\C$, number theorists would like to have 
parametrizations for which the coefficients of the rational functions
lie in fields of small degree, preferably in the field of rational
numbers.} if and only if its 
genus is $0$; he also showed in \cite{Clell} that curves of genus $1$ 
can be parametrized by elliptic functions\footnote{This result later
gave elliptic curves their name. The fact that the rational points on
elliptic curves form a group was pointed out by Juel \cite{Juel} and
Mordell \cite{Mord1}. Clebsch \cite{ClWei} already pointed out that
if $P_1$, $P_2$, $P_3$ are collinear points on a cubic curve parametrized
by an elliptic function $f$, then the corresponding arguments $z_1$, 
$z_2$, $z_3$ of $f$ have the property that $z_1 + z_2 + z_3$ is
constant up to multiples of the periods of $f$.}. 
Kapferer\footnote{Heinrich Kapferer was 
born on October 28, 1888 in Donaueschingen (Bavaria). He studied 
at the University of Freiburg, with a short visit to Munich 
for one semester. His Ph.D. thesis  in Freiburg (1917) was supervised 
by Stickelberger. Kapferer worked as a teacher from 1914 to 1924; 
in 1922, he took up his studies at the Universities of G\"ottingen 
and Freiburg and received his habilitation (the right to teach at
a university -- venia legendi) in 1926. In 1932, Kapferer got an 
appointment as a professor at Freiburg, but his position was cancelled 
in 1937 despite support by S\"uss and Hasse.
Kapferer worked at the University library until 1941, when he was
forced to ''take a leave''. He died on January 5, 1984 in Freiburg.} 
\cite{Kap} observed that the following special case of Clebsch's 
theorem on the rational parametrization of curves can be proved quite
easily\footnote{Shafarevich's proof in  \cite{Sha} that the Fermat 
curve $X^n + Y^n = Z^n$ for $n > 2$ cannot be parametrized is 
nothing but Kapferer's proof in this special case.}:

\begin{thm}
Let $K$ be a field with characteristic $0$. 
Let $\cC$ be a nonsingular curve defined by the homogeneous
polynomial $F(X,Y,Z) \in K[X,Y,Z]$ of degree $n \ge 1$. If $\cC$ 
can be parametrized, that is, if there exist nonconstant coprime 
homogeneous polynomials $f, g, h \in K[U,V]$ of degree $m \ge 1$ 
such that $F(f,g,h) = 0$ identically, then $n \le 2$.
\end{thm}

In particular, elliptic curves (smooth cubic curves with at least one
rational point) cannot be parametrized with rational functions. 

\begin{proof}
Assume that there is a parametrization $F(f,g,h) = 0$ as described 
in the statement of the Theorem. Taking the derivatives of this 
equation with respect to $U$ and $V$ we find
\begin{align*}
 F_X(f,g,h) f_U + F_Y(f,g,h) g_U + F_Z(f,g,h) h_U & = 0, \\
 F_X(f,g,h) f_V + F_Y(f,g,h) g_V + F_Z(f,g,h) h_V & = 0.
\end{align*}
These equations can be written in matrix form
\begin{equation}\label{EMKa}
 \bigg(\begin{matrix} f_U & g_U & h_U \\ f_V & g_V & h_V 
        \end{matrix} \bigg) 
   \left( \begin{matrix} F_X(f,g,h) \\ F_Y(f,g,h) \\ F_Z(f,g,h)
          \end{matrix} \right) = 0. 
\end{equation}
We now start with two little lemmas:

\begin{lem}
The $2 \times 3$-matrix in (\ref{EMKa}) has rank $2$.
\end{lem}

If not, then its three columns are linearly dependent, hence its 
three minors vanish. But $f_U g_V - f_V g_U = 0$, together with 
Euler's identities\footnote{By linearity, it is sufficient to prove
these equations for monomials $f(U,V) = U^rV^s$; in this case, the
identities are easily checked.} 
$$ mf = Uf_U + Vf_V, \qquad mg = Ug_U + Vg_V, $$
implies $U (f_U g_V - f_V g_U) = m(f \cdot g_V - g \cdot f_V) = 0$,
hence $f \cdot g_V - g \cdot f_V = 0$. Since $f$ and $g$ were
assumed to be coprime\footnote{In fact, any common divisor of
$f$ and $g$ can either be cancelled (if $F$ does not contain
any monomial of the form $Z^n$) or it divides $h$, contradicting 
the assumption that $f, g, h$ be coprime.}, this implies 
$f \mid f_V$ and hence $f_V = 0$ (the same argument, by the way, 
occurs in Snyder's proof of Mason's theorem). Similarly, it 
follows that $f_U = 0$. Since $K$ has characteristic $0$, this is 
only possible if $\deg f \le 0$, which contradicts our assumptions. 

\begin{lem}
The polynomials $F_X(f,g,h)$, $F_Y(f,g,h)$ and $F_Z(f,g,h)$ are coprime.
\end{lem}

In fact, assume not; then, over some algebraic closure of $K$ they 
will have at least a linear factor $U - cV$ in common. Setting
$x = f(c,1)$, $y = g(c,1)$ and $z = h(c,1)$ we have 
$(x,y,z) \ne (0,0,0)$: otherwise $f, g, h$ would have a common factor
$U - cV$ contrary to our assumptions. Moreover, we have 
$F_X(x,y,z) = F_Y(x,y,z) = F_Z(x,y,z) = 0$, and this implies
that $(x:y:z)$ is a singular point on $\cC$. 

\medskip

Now we can complete the proof of the theorem. 
We now know that the solution space of the linear system of equations
\begin{equation}\label{EMKb}
  \bigg(\begin{matrix} f_U & g_U & h_U \\ f_V & g_V & h_V 
        \end{matrix} \bigg) 
   \left( \begin{matrix} P \\ Q \\ R \end{matrix} \right) = 0 
\end{equation}
in the $3$-dimensional $K(U,V)$-vector space has dimension $1$. 
Developing the determinants of the matrices 
$$ M = \left(\begin{matrix} 
        f_U & g_U & h_U \\ f_U & g_U & h_U \\  f_V & g_V & h_V 
        \end{matrix} \right) \quad \text{and} \quad
   N =  \left(\begin{matrix} 
        f_V & g_V & h_V \\ f_U & g_U & h_U  \\ f_V & g_V & h_V 
        \end{matrix} \right)   $$
with respect to the first line shows that
$0 = \det M = f_U P + g_U Q + h_U R$ and $0 = \det N = f_V P + g_V Q + h_V R$, 
where
$$ P =   \bigg| \begin{matrix} g_U & h_U \\ g_V & h_V  \end{matrix} \bigg|, 
  \qquad
   Q = - \bigg| \begin{matrix} f_U & h_U \\ f_V & h_V  \end{matrix} \bigg|, 
  \qquad
   R =  \bigg| \begin{matrix} f_U & g_U \\ f_V & g_V  \end{matrix} \bigg|. $$ 
This provides us with a solution of (\ref{EMKb}).
Since $P = F_X(f,g,h)$, $Q = F_Y(f,g,h)$ and $R = F_Z(f,g,h)$ is another
solution, the fact that the solution space has dimension $1$ implies that 
these solutions are linearly dependent over $K(U,V)$. Thus 
there exists a rational function $\frac{p(U,V)}{q(U,V)}$ with coprime 
polynomials 
$p, q \in K[U,V]$ such that 
\begin{align*}
   q(U,V) (h_U g_V - h_V g_U) & = p(U,V) F_X(f,g,h), \\
   q(U,V) (f_U h_V - f_V h_U) & = p(U,V) F_Y(f,g,h), \\
   q(U,V) (g_U f_V - g_V f_U) & = p(U,V) F_Z(f,g,h).
\end{align*}
These equations imply that $q(U,V) = q$ is a constant: in fact, each 
irreducible factor of $q$ must divide $F_X(f,g,h)$, $F_Y(f,g,h)$, 
and $F_Z(f,g,h)$, hence their gcd, which is trivial by our second claim. 

Next we compute degrees; on the left hand side we get
$$ \deg q(h_Ug_V - h_Vg_U) = \deg (h_Ug_V - h_Vg_U) 
                          \le \deg h_Ug_V = 2(m-1). $$ 
On the right hand side, we find
$$ \deg p \cdot F_X(f,g,h) \ge \deg F_X(f,g,h) = (n-1)m. $$ 
Comparing degrees now shows that $2m-2 \ge m(n-1)$, which 
implies $-2 \ge m(n-3)$ and therefore $n < 3$.
\end{proof}

This is not the best possible result that can be achieved by this
line of attack. It is easy to prove that irreducible curves with 
a single double point cannot be parametrized by rational functions
whenever the curve has degree $\ge 4$. Gradually generalizing this
proof will ultimately lead to a proof of Clebsch's result that
a curve cannot be parametrized if its genus is positive.

\section*{Acknowledgements}
I would like to thank the referees for carefully reading the 
manuscript and suggesting many improvements.


\begin{thebibliography}{}

\bibitem{AL} W. Aitken, F. Lemmermeyer,
{\em Counterexamples to the Hasse principle: an elementary introduction},
Amer. Math. Monthly (2011), to appear
%

\bibitem{ClWei} A. Clebsch,
{\em \"Uber einen Satz von Steiner und einige Punkte der Theorie
     der Curven dritter Ordnung},
J. Reine Angew. Math. {\bf 63} (1863), 94--121
%

\bibitem{Clrat} A. Clebsch,
{\em \"Uber diejenigen ebenen Kurven, deren Koordinaten rationale
     Funktionen eines Parameters sind},
     J. Reine Angew. Math. {\bf 64} (1865), 43--65
%

\bibitem{Clell} A. Clebsch,
{\em \"Uber diejenigen Kurven, deren Koordinaten sich als
elliptische Funktionen eines Parameters darstellen lassen},
J. Reine Angew. Math. {\bf 64} (1865), 210--270
%

\bibitem{Clgen} A. Clebsch,
{\em Vorlesungen \"uber die Geometrie},
lecture notes 1871/72, Lindemann (ed.)
%

\bibitem{Ful} W. Fulton,
{\em Algebraic Curves},
New York - Amsterdam, Benjamin 1969
%

\bibitem{HH} D. Hilbert, A. Hurwitz,
{\em \"Uber die diophantischen Gleichungen vom Geschlecht Null},
Acta Math. {\bf 14} (1891), 217--224
%

\bibitem{Juel} C. Juel,
{\em Ueber die Parameterbestimmung von Punkten auf
Curven zweiter und dritter Ordnung. Eine geometrische
Einleitung in die Theorie der logarithmischen und
elliptischen Funktionen}, Math. Ann. {\bf 47} (1896), 72--104
%

\bibitem{Kap} H. Kapferer,
{\em \"Uber das Kriterium der Rationalit\"at einer algebraischen Kurve},
       Sitz.-ber. M\"unchen (1930), 123--128
%

\bibitem{Kro} L. Kronecker,
{\em Vorlesungen \"uber Zahlentheorie}, (K. Hensel, ed.)
Teubner, Leipzig (1901); reprint Springer-Verlag (1978)
%

\bibitem{LPep} F. Lemmermeyer,
{\em A note on P\'epin's counterexamples to the Hasse principle for
    curves of genus 1}, 
Abh. Math. Semin. Univ. Hamb. {\bf 69} (1999), 335--345 
%

\bibitem{Mord1} L.J. Mordell
{\em On the rational solutions of the indeterminate equations of the 
     third and fourth degrees},
Cambr. Phil. Soc. Proc. {\bf 21} (1922), 179--192
%

\bibitem{New} I. Newton,
{\em The Mathematical Papers of Isaac Newton},
vol. 4, (D.T. Whiteside, ed.), Cambridge 1971
%

\bibitem{Poi} H. Poincar\'e,
{\em Sur les propri\'et\'es arithm\'etiques des courbes alg\'ebriques},
J. Math. (5) {\bf 7} (1901), 161--233; \OE uvres 5 (1950), 483--548
%

\bibitem{Reid} M. Reid,
{\em Undergraduate algebraic geometry},
LMS Student Texts 12, CUP 1988
%

\bibitem{Samuel} P. Samuel,
{\em Projective geometry}, Springer-Verlag 1988
%

\bibitem{SWP} J.R. Sendra, F. Winkler, S. P\'erez-D\'iaz,
{\em Rational Algebraic Curves - A Computer Algebra Approach}, 
Springer-Verlag 2008
%

\bibitem{Sha} I.R. Shafarevich,
{\em Basic Algebraic Geometry. 1: Varieties in projective space},
Transl. from the Russian by Miles Reid. 2nd ed. Springer-Verlag (1994)
%

\bibitem{ShaCl}  I.R. Shafarevich,
{\em Zum 150. Geburtstag von Alfred Clebsch},
Math. Ann. {\bf 266} (1869), 135--140
%

\bibitem{Sny} N. Snyder,
{\em  An alternate proof of Mason's theorem},
Elem. Math. {\bf 55} (2000), 93--94
%

\end{thebibliography}
\end{document}